\documentclass[A4]{article}
\usepackage{amscd}
\usepackage{amssymb}
\usepackage{amsmath}
\usepackage{amsthm}
\usepackage{latexsym}
\usepackage{upref}
\usepackage{epsfig}
\usepackage{mathrsfs}
\usepackage{float}
\usepackage{indentfirst}
\usepackage{hyperref}
\hypersetup{
    colorlinks=true,
    linkcolor=blue,
    citecolor=blue,
    filecolor=blue,
    urlcolor=blue,
}
\usepackage{graphics,graphicx,color}
\usepackage{floatflt}
\usepackage{cite}
\usepackage{enumitem}
\usepackage{fancyhdr} 
\usepackage{mathptmx} 

\renewcommand\footnotemark{}

\theoremstyle{plain}
\newtheorem{theorem}{Theorem}[section]
\newtheorem{lemma}[theorem]{Lemma}
\newtheorem{proposition}[theorem]{Proposition}

\newtheorem{corollary}[theorem]{Corollary}
\theoremstyle{definition}
\newtheorem{definition}[theorem]{Definition}

\theoremstyle{remark}

\makeatletter
\def\th@plain{%
  \thm@notefont{}
  \itshape 
}
\def\th@definition{%
  \thm@notefont{}
  \normalfont 
} \makeatother

\setlist{font=\normalfont}

\DeclareMathAlphabet{\mathcal}{OMS}{cmsy}{m}{n} %

\newcommand{\aut}[1]{\mathrm{Aut}\,{(#1)}}
\newcommand{\gyr}[2]{\mathrm{gyr}\,{[{#1}]{#2}}}
\newcommand{\set}[1]{\{#1\}}
\newcommand{\cset}[2]{\set{{#1}\colon{#2}}}
\newcommand{\abs}[1]{|#1|}
\newcommand{\sym}[1]{\mathrm{Sym}\,{(#1)}}
\newcommand{\id}[1]{\mathrm{id}_{#1}}
\newcommand{\orb}[1]{\mathrm{orb}\,{#1}}
\newcommand{\stab}[1]{\mathrm{stab}\,{#1}}
\newcommand{\lcap}[2]{\displaystyle\bigcap_{#1}^{#2}}
\newcommand{\Fix}[1]{\mathrm{Fix}\,{(#1)}}
\newcommand{\fix}[1]{\mathrm{fix}\,{#1}}
\newcommand{\lsum}[2]{\displaystyle\sum_{#1}^{#2}}
\newcommand{\Abs}[1]{\Big|#1\Big|}
\newcommand{\lcup}[2]{\displaystyle\bigcup_{#1}^{#2}}
\newcommand{\Bp}[1]{\left(#1\right)}
\newcommand{\conj}[2]{\mathrm{conj}_{#1}\,{(#2)}}
\newcommand{\tranc}[1]{\mathcal{TC}\,{(#1)}}
\newcommand{\nulll}{\hskip4pt}

\newcommand{\symz}[1]{\mathrm{Sym}_0\,{(#1)}}
\newcommand{\B}{\mathbb{B}}
\newcommand{\R}{\mathbb{R}}
\newcommand{\norm}[1]{\|#1\|}
\newcommand{\gen}[1]{\langle#1\rangle}

\newcommand{\Bb}[1]{\left\{#1\right\}}

\newcommand{\vphi}{\varphi}
\newcommand{\sig}{\sigma}
\newcommand{\Gam}{\Gamma}
\newcommand{\bet}{\beta}
\newcommand{\alp}{\alpha}
\newcommand{\lamb}{\lambda}
\newcommand{\gam}{\gamma}

\renewcommand{\iff}{\Leftrightarrow}
\renewcommand{\vec}[1]{\mathbf{#1}}

\begin{document}
\title{\textbf{Gyrogroup actions:\\ A generalization of group actions}$^{\ast,\dag}$\footnote{$^\ast$This is the final version of the
manuscript appeared in J. Algebra {\bf 454}~(2016), 70--91. The
published version of the article is accessible via
doi:\href{http://dx.doi.org/10.1016/j.jalgebra.2015.12.033}{{10.1016/j.jalgebra.2015.12.033}}.
}\footnote{$^\dag$Part of this work has been presented at the Joint
Mathematics Meetings (JMM 2016), Seattle, WA, USA, January 6--9,
2016.}}
\author{Teerapong Suksumran$^{\ddag}$\footnote{$^\ddag$The author was financially supported by Institute for Promotion of Teaching Science and
Technology (IPST), \mbox{Thailand}, through Development and
Promotion of Science and Technology Talents Project
(DPST).}\footnote{\copyright\,2016 Author. This manuscript version
is made available under the CC-BY-NC-ND 4.0 license
(\href{http://creativecommons.org/licenses/by-nc-nd/4.0/}{{http://creativecommons.org/licenses/by-nc-nd/4.0/}}).}\\[5pt]
Department of Mathematics\\
North Dakota State University\\
Fargo, ND 58105, USA\\[5pt]
\texttt{teerapong.suksumran@gmail.com}}
\date{}
\maketitle

\begin{abstract}
This article explores the novel notion of gyrogroup actions, which
is a natural generalization of the usual notion of group actions. As
a first step toward the study of gyrogroup actions from the
algebraic viewpoint, we prove three well-known theorems in group
theory for gyrogroups: the orbit-stabilizer theorem, the orbit
decomposition theorem, and the Burnside lemma (or the
Cauchy-Frobenius lemma). We then prove that under a certain
condition, a gyrogroup $G$ acts transitively on the set $G/H$ of
left cosets of a subgyrogroup $H$ in $G$ in a natural way. From this
we prove the structure theorem that every transitive action of a
gyrogroup can be \mbox{realized} as a gyrogroup action by left
gyroaddition. We also exhibit concrete \mbox{examples} of gyrogroup
actions from the M\"{o}bius and Einstein gyrogroups.
\end{abstract}
\thispagestyle{empty}

\section{Introduction}\label{sec: introdcution}

The method of group action is considered as an important and a
powerful tool in mathematics. It is used to unravel many of
mathematical structures. In fact, when one structure acts on another
structure, a better understanding is obtained on both. The Sylow
theorems, for instance, result from the action of a group on itself
by conjugation. The algebraic structure of a finite group in turn is
revealed by the use of the Sylow theorems. Further, many
wide-application theorems in group theory and combinatorics such as
the orbit-stabilizer theorem, the class equation, and the Burnside
lemma (also called the Cauchy-Frobenius lemma) are byproducts of
group actions. In addition, group actions are used to prove versions
of the Lagrange theorem, the Sylow theorems, and the Hall theorems
for Bruck loops in \cite{BBAS2011FBL}.

There are some attempts to study permutation representations of
quasigroups and loops---algebraic structures that are a
generalization of groups. For instance, Jonathan Smith has
intensively studied quasigroup and loop representations in a series
of papers \cite{JS2003PRL, JS2006PRL, JS2007FLQ}. Also, the study of
sharply transitive sets in quasigroup actions can be found in
\cite{BIJRJS2011STS}.

Gyrogroups are a suitable generalization of groups, arising from the
study of the parametrization of the Lorentz transformation group by
Abraham Ungar \cite{AU1988TRP}. The origin of a gyrogroup is
described in \cite{AU2008AHG} and references therein. Gyrogroups
share \mbox{remarkable} analogies with groups. In fact, every group
forms a gyrogroup under the same operation. Many of classical
theorems in group theory continue to hold for gyrogroups, including
the Lagrange theorem \cite{TSKW2014LTG}, the fundamental isomorphism
\mbox{theorems} \cite{TSKW2015ITG}, and the Cayley theorem
\cite{TSKW2015ITG}. Further, harmonic analysis can be studied in the
framework of gyrogroups \cite{MF2014HAE, MF2015HAM}, using the
gyroassociative law in place of the associative law.

In this article, we continue to prove some well-known theorems in
group theory for gyrogroups, including the orbit-stabilizer theorem,
the Burnside lemma, and the orbit decomposition theorem. These
theorems are proved by techniques similar to those used in group
theory, where gyroautomorphisms play the fundamental role and the
associative law is replaced by the gyroassociative law. We recover
the classical results if a gyrogroup is {\it degenerate} in the
sense that its gyroautomorphisms are the identity automorphism. It
is worth pointing out that the results involving gyrogroups may be
recast in the framework of left Bol loops with the property that
left inner mapping are automorphisms.

The structure of the article is organized as follows. In Section
\ref{sec: gyrogroup}, we review the basic theory of gyrogroups. In
Section \ref{sec: gyrogroup action}, we study gyrogroup actions or,
equivalently, permutation representations of a gyrogroup. In Section
\ref{sec: left gyroaddition}, we prove that under a \mbox{certain}
condition, a gyrogroup $G$ acts transitively on the coset space
$G/H$ of left cosets of a subgyrogroup $H$ in $G$ by left
gyroaddition. This results in the fundamental iso-morphism theorem
for gyrogroup actions. In Section \ref{sec: example}, we exhibit
concrete examples of transitive gyrogroup actions from the
M\"{o}bius and Einstein gyrogroups.

\section{Gyrogroups}\label{sec: gyrogroup}
\par We summarize basic properties of gyrogroups for reference. Much of
this section can be found in \cite{AU2008AHG, TSKW2014LTG,
TSKW2015ITG}. The reader familiar with gyrogroup theory may skip
this section.

\par A pair $(G,\oplus)$ consisting of a nonempty set $G$ and a binary
operation $\oplus$ on $G$ is called a \textit{groupoid}. The group
of automorphisms of a groupoid $(G, \oplus)$ is denoted by
$\aut{G,\oplus}$. Ungar \cite{AU2008AHG} formulates the formal
definition of a gyrogroup as follows.
\begin{definition}[Gyrogroups]\label{def: gyrogroup}
A groupoid $(G,\oplus)$ is a \textit{gyrogroup} if its binary
operation satisfies the following axioms.
\begin{enumerate}
    \item[(G1)] There is an element $0\in G$ such that $0\oplus a =
    a$ for all $a\in G$.
    \item[(G2)] For each $a\in G$, there is an element $b\in G$ such that
$b\oplus a = 0$.
    \item[(G3)] For all $a$, $b\in G$, there is an automorphism
$\gyr{a,b}{}\in\aut{G,\oplus}$ such that
    \begin{equation}\tag{left gyroassociative law} a\oplus (b\oplus c) = (a\oplus b)\oplus\gyr{a,
    b}{c}\end{equation}
    for all $c\in G$.
    \item[(G4)] For all $a$, $b\in G$, $\gyr{a\oplus b, b}{} = \gyr{a,
    b}{}$.\hfill(left loop property)
\end{enumerate}
\end{definition}

\par We remark that the axioms in Definition \ref{def: gyrogroup} imply
the right counterparts. In particular, any gyrogroup has a unique
two-sided identity $0$, and an element $a$ of the gyro-group has a
unique two-sided inverse $\ominus a$. Let $G$ be a gyrogroup. For
$a, b\in G$, the map $\gyr{a, b}{}$ is called the
\textit{gyroautomorphism generated by $a$ and $b$}. By Theorem 2.10
of \cite{AU2008AHG}, the gyroautomorphisms are completely determined
by the \textit{gyrator identity}
\begin{equation}\label{eqn: gyrator identity}
\gyr{a, b}{c} = \ominus(a\oplus b)\oplus(a\oplus (b\oplus c))
\end{equation}
for all $a$, $b$, $c\in G$. Note that every group forms a gyrogroup
under the same operation by defining the gyroautomorphisms to be the
identity automorphism, but the converse is not in general true. From
this point of view, gyrogroups suitably generalize groups.

\par Recall that the \textit{coaddition}, $\boxplus$, of a gyrogroup
$G$ is defined by the equation
\begin{equation}\label{eqn: gyrogroup coaddition}
a\boxplus b = a\oplus\gyr{a, \ominus b}{b}
\end{equation}
for $a, b\in G$. Set $a\boxminus b = a\boxplus (\ominus b)$ for $a,
b\in G$. By Theorem 2.22 of \cite{AU2008AHG}, the following
cancellation laws hold in gyrogroups.
\begin{theorem}[{\hskip-0.2pt}\cite{AU2008AHG}]\label{thm: cancellation law in gyrogroup}
Let $G$ be a gyrogroup. For all $a$, $b$, $c\in G$,
\begin{enumerate}
    \item $a\oplus b = a\oplus c$ implies $b = c$;\hfill\normalfont{(general left cancellation
    law)}
    \item $\ominus a\oplus(a\oplus b) = b$;\hfill\normalfont{(left cancellation law)}
    \item $(b\ominus a)\boxplus a = b$;\hfill\normalfont{(right cancellation law I)}
    \item $(b\boxminus a)\oplus a = b$. \hfill\normalfont{(right cancellation law II)}
\end{enumerate}
\end{theorem}

\par Subgyrogroups, gyrogroup homomorphisms, and quotient gyrogroups
are studied in \cite{TSKW2014LTG, TSKW2015ITG}. An attempt to prove
a gyrogroup version of the Lagrange theorem leads to the notion of
L-subgyrogroups \cite{TSKW2015ITG}. A subgyrogroup $H$ of a
gyrogroup $G$ is called an \textit{L-subgyrogroup} if $\gyr{a,
h}{(H)} = H$ for all $a\in G$ and $h\in H$. L-subgyrogroups behave
well in the sense that they partition $G$ into left cosets of equal
size. More precisely, if $H$ is an L-subgyrogroup of $G$, then $H$
and $a\oplus H := \cset{a\oplus h}{h\in H}$ have the same
cardinality and the coset space $G/H := \cset{a\oplus H}{a\in G}$
forms a \textit{disjoint} partition of $G$. In the case where $G$ is
a finite gyrogroup, we obtain the familiar index formula
\begin{equation}\label{eqn: index formula}
\abs{G} = [G\colon H]\abs{H}.
\end{equation}
Here, $[G\colon H]$ denotes the \textit{index} of $H$ in $G$, which
is defined as the cardinality of the coset space $G/H$. We will see
in Section \ref{sec: left gyroaddition} that certain L-subgyrogroups
give rise to gyrogroup actions by left gyroaddition.

\par For any \textit{non}-L-subgyrogroup $K$ of a gyrogroup $G$, it is no
longer true that the left cosets of $K$ partition $G$. Moreover, the
index formula $\abs{G} = [G\colon K]\abs{K}$ is not in general true.
Nevertheless, the order of any subgyrogroup of a finite gyrogroup
$G$ divides the order of $G$, see \cite[Theorem 5.7]{TSKW2014LTG}.

\section{Gyrogroup actions}\label{sec: gyrogroup action}
\par Throughout the remainder of the article, $X$ is a (finite or
infinite) nonempty set and $G$ is a (finite or infinite) gyrogroup
unless otherwise stated.

\subsection{Definition}

\begin{definition}[Gyrogroup actions]\label{def: action}
A map from $G\times X$ to $X$, written $(a, x)\mapsto a\cdot x$, is
a (gyrogroup) \textit{action} of $G$ on $X$ if the following
conditions hold:
\begin{enumerate}
    \item $0\cdot x = x$ for all $x\in X$, and
    \item $a\cdot(b\cdot x) = (a\oplus b)\cdot x$ for all $a$, $b\in
    G$, $x\in X$.
\end{enumerate}
In this case, $X$ is said to be a \textit{$G$-set} and $G$ is said
to act on $X$.
\end{definition}

\par Note that if $G$ is a gyrogroup with trivial gyroautomorphisms, then
the notion of gyrogroup actions specializes to the usual notion of
group actions. In other words, gyrogroup actions naturally
generalize group actions. We will present some concrete examples of
gyrogroups that satisfy the axioms of a gyrogroup action in Section
\ref{sec: example}.

\par Let $\sym{X}$ denote the group of permutations of $X$. Since
every group is a gyro-group under the same operation, $\sym{X}$,
together with function composition $\circ$, forms a gyrogroup with
trivial gyroautomorphisms. Thus, it makes sense to speak of a
gyro-group homomorphism from $G$ to $\sym{X}$. Recall that a map
$\vphi\colon G\to\sym{X}$ is a \textit{gyrogroup homomorphism} if
\begin{equation}
\vphi(a\oplus b) = \vphi(a)\circ\vphi(b)
\end{equation}
for all $a, b\in G$.

\par The following theorem asserts that any gyrogroup action of $G$ on
$X$ induces a gyro-group homomorphism from $G$ to $\sym{X}$. Let $X$
be a $G$-set. For each $a\in G$, define $\sig_a$ by
\begin{equation}\label{eqn: Permutation of X induced by action}
\sig_a(x) = a\cdot x,\qquad x\in X.
\end{equation}
Define $\dot\vphi$ by the equation
\begin{equation}\label{eqn: Action as a homomorphism}
\dot\vphi(a) = \sig_a,\qquad a\in G.
\end{equation}

\newpage

\begin{theorem}\label{thm: Action as a homomorphism}
Let $X$ be a $G$-set.
\begin{enumerate}
    \item\label{item: Sigma_a is a permutation} For each $a\in G$, $\sig_a$ defined by \eqref{eqn: Permutation of X
induced by action} is a permutation of $X$.
    \item\label{item: Varphidot is a homomorphism} The map $\dot\vphi$ defined by \eqref{eqn: Action as a
    homomorphism} is a gyrogroup homomorphism from $G$ to $\sym{X}$. Its kernel is
    $\ker{\dot\vphi} = \cset{a\in G}{a\cdot x = x \,\textrm{ for all }x\in X}$.
\end{enumerate}
\end{theorem}
\begin{proof}
The theorem follows directly from the axioms of a gyrogroup action.
Note that $a\in\ker{\dot\vphi}$ if and only if $\sig_a = \id{X}$ if
and only if $a\cdot x = x$ for all $x\in X$.
\end{proof}

\par According to Theorem \ref{thm: Action as a homomorphism},
$\dot\vphi$ is called the \textit{gyrogroup homomorphism
\mbox{afforded} by a gyrogroup action of $G$ on $X$} or the
\textit{associated permutation representation of $G$}. The process
of turning a gyrogroup action into a permutation representation is
reversible in the sense of the following theorem.
\begin{theorem}\label{thm: action from permutation representation}
Let $\vphi\colon G\to\sym{X}$ be a gyrogroup homomorphism. The map
$\cdot$ defined by the equation
\begin{equation}
a\cdot x = \vphi(a)(x),\qquad a\in G,\, x\in X,
\end{equation}
is a gyrogroup action of $G$ on $X$. Furthermore, $\dot\vphi =
\vphi$.
\end{theorem}
\begin{proof}
Clearly, $0\cdot x = \vphi(0)(x) = \id{X}(x) = x$ for all $x\in X$.
For all $a$, $b\in G$, $x\in X$, we have $a\cdot(b\cdot x) =
(\vphi(a)\circ\vphi(b))(x) = \vphi(a\oplus b)(x) = (a\oplus b)\cdot
x$. By construction, $\dot\vphi(a)(x) = a\cdot x = \vphi(a)(x)$ for
all $x\in X$, $a\in G$. Hence, $\dot\vphi = \vphi$.
\end{proof}

\par Theorems \ref{thm: Action as a homomorphism} and \ref{thm: action
from permutation representation} together imply that the study of
gyrogroup actions of $G$ on $X$ is equivalent to the study of
gyrogroup homomorphisms from $G$ to $\sym{X}$.

\subsection{Orbits and stabilizers}
\par In this section, we prove gyrogroup versions of three
well-known theorems in group theory and combinatorics:
\begin{itemize}\renewcommand{\labelitemi}{$\bullet$}
    \item the orbit-stabilizer theorem;
    \item the orbit decomposition theorem;
    \item the Burnside lemma.
\end{itemize}

\par We will see shortly that {\it stabilizer subgyrogroups} of $G$ have
nice properties. For instance, they share remarkable properties with
stabilizer subgroups; are L-subgyrogroups and hence partition $G$
into left cosets; and are invariant under the gyroautomorphisms of
$G$. Among other things, they lead to the orbit-stabilizer theorem
for gyrogroups.

\par Let $X$ be a $G$-set. Define a relation $\sim$ on $X$ by the
condition
\begin{equation}\label{eqn: Equivalence relation from action}
x\sim y \quad\Leftrightarrow\quad y = a\cdot x \textrm{ for some
}a\in G.
\end{equation}

\begin{theorem}\label{thm: Equivalence relation from action}
The relation $\sim$ defined by \eqref{eqn: Equivalence relation from
action} is an equivalence relation on $X$.
\end{theorem}
\begin{proof}
Let $x$, $y$, $z\in X$.

\textit{Reflexive.} Since $0\cdot x = x$, we have $x\sim x$.

\textit{Symmetric.} Suppose that $x\sim y$. Then $y = a\cdot x$ for
some $a\in G$. Since $(\ominus a)\cdot y = x$, we have $y\sim x$.

\textit{Transitive.} Suppose that $x\sim y$ and $y\sim z$. Then $y =
a\cdot x$ and $z = b\cdot y$ for some \mbox{$a$, $b\in G$}. Since $z
= b\cdot (a\cdot x) = (b\oplus a)\cdot x$, we have $x\sim z$.
\end{proof}

\par Let $x\in X$. The equivalence class of $x$ determined by the
relation $\sim$ is called the \textit{orbit of $x$} and is denoted
by $\orb{x}$, that is, $\orb{x} = \cset{y\in X}{y\sim x}$. It is
straightforward to check that
\begin{equation}\label{eqn: orbit}
\orb{x} = \cset{a\cdot x}{a\in G}
\end{equation}
for all $x\in X$. The \textit{stabilizer of $x$ in $G$}, denoted by
$\stab{x}$, is defined as
\begin{equation}\label{eqn: stabilizer subgyrogroup}
\stab{x} = \cset{a\in G}{a\cdot x = x}.
\end{equation}

\begin{proposition}\label{prop: Stabilizer subgyrogroup}
Let $X$ be a $G$-set.
\begin{enumerate}
    \item\label{item: stab x is a subgyrogroup} For each $x\in X$, $\stab{x}$ is a subgyrogroup of $G$.
    \item\label{item: Kerner of varphidot as intersection} $\ker{\dot\vphi} = \lcap{x\in
    X}{}\stab{x}$.
\end{enumerate}
\end{proposition}
\begin{proof}
Clearly, $0\in \stab{x}$. Let $a$, $b\in \stab{x}$. Since $(a\oplus
b)\cdot x = a\cdot (b\cdot x) = x$, we have $a\oplus b\in\stab{x}$.
Also, $(\ominus a)\cdot x = (\ominus a)\cdot (a\cdot x) = x$. Hence,
$\ominus a\in\stab{x}$. By the subgyrogroup criterion
\cite[Proposition 14]{TSKW2015ITG}, $\stab{x}\leqslant G$. Item
\eqref{item: Kerner of varphidot as intersection} follows directly
from Theorem \ref{thm: Action as a homomorphism} \eqref{item:
Varphidot is a homomorphism}.
\end{proof}

\begin{proposition}\label{prop: Stabilizer invariant under gyroautomorphism}
Let $X$ be a $G$-set. For all $a$, $b\in G$, $x\in X$,
$$\gyr{a, b}{(\stab{x})} \subseteq \stab{x}.$$
In particular, if $c\in G$ and $c\cdot x = x$, then $(\gyr{a,
b}{c})\cdot x = x$.
\end{proposition}
\begin{proof}
Let $c\in\stab{x}$. According to the gyrator identity \eqref{eqn:
gyrator identity}, we compute
\begin{align}
(\gyr{a, b}{c})\cdot x &= [\ominus(a\oplus b)\oplus(a\oplus (b\oplus
c))]\cdot x\notag\\
{} &= \ominus(a\oplus b)\cdot[a\cdot (b\cdot (c\cdot x))]\notag\\
{} &= \ominus(a\oplus b)\cdot[a\cdot (b\cdot x)]\notag\\
{} &= \ominus(a\oplus b)\cdot[(a\oplus b)\cdot x]\notag\\
{} &= x.\notag
\end{align}
Hence, $\gyr{a, b}{c}\in\stab{x}$, which completes the proof.
\end{proof}

\begin{corollary}\label{cor: Stabilizer invariant under gyroautomorphism}
Let $X$ be a $G$-set. For all $a$, $b\in G$, $x\in X$,
\begin{enumerate}
    \item\label{item: Invariant of stabilizer} $\gyr{a, b}{(\stab x)}
    = \stab{x}$, and
    \item\label{item: Stabilizer is an L-subgyrogroup} $\stab{x}$ is an L-subgyrogroup of $G$.
\end{enumerate}
\end{corollary}
\begin{proof}
Item \eqref{item: Invariant of stabilizer} follows directly from
Proposition 6 of \cite{TSKW2015ITG}. By item \eqref{item: Invariant
of stabilizer}, $\stab{x}$ is invariant under the gyroautomorphisms
of $G$. Hence, $\stab{x}\leqslant_L G$.
\end{proof}

\begin{lemma}\label{lem: Equivalence of equal cosets of stabilizer}
Let $X$ be a $G$-set. For all $a$, $b\in G$, $x\in X$, the following
are equivalent:
\begin{enumerate}
    \item\label{item: ax = bx} $a\cdot x = b\cdot x$;
    \item\label{item: (-b+a)x = x} $(\ominus b\oplus a)\cdot x = x$;
    \item\label{item: Equal cosets} $a\oplus \stab{x} = b\oplus\stab{x}$.
\end{enumerate}
\end{lemma}
\begin{proof}
The proof of the equivalence \eqref{item: ax = bx} $\Leftrightarrow$
\eqref{item: (-b+a)x = x} is straightforward, using the axioms of a
gyrogroup action. By Corollary \ref{cor: Stabilizer invariant under
gyroautomorphism}, $\stab{x}\leqslant_L G$. It follows that
$$(\ominus b\oplus a)\cdot x = x \quad\iff\quad \ominus b\oplus a\in\stab{x}\quad\iff\quad a\oplus \stab{x} = b\oplus\stab{x}.$$
This proves the equivalence \eqref{item: (-b+a)x = x}
$\Leftrightarrow$ \eqref{item: Equal cosets}.
\end{proof}

\par We are now in a position to state a gyrogroup version of the orbit-stabilizer theorem.

\begin{theorem}[The orbit-stabilizer theorem]\label{thm: Orbit-stabilizer theorem}
Let $G$ be a gyrogroup acting on a set $X$. For each $x\in X$, there
exists a bijection from the orbit of $x$ to the coset space
$G/\stab{x}$. In particular, if $G$ is a finite gyrogroup, then
\begin{equation}\label{eqn: orbit-stabilizer theorem}
\abs{G} = \abs{\orb{x}}\abs{\stab{x}}.
\end{equation}
\end{theorem}
\begin{proof}
Let $\theta$ be the map defined on $\orb{x}$ by
$$\theta(a\cdot x) = a\oplus\stab{x},\qquad a\in G.$$
By Lemma \ref{lem: Equivalence of equal cosets of stabilizer},
$\theta$ is well defined and injective. That $\theta$ is surjective
is clear. Since $\stab{x}$ is an L-subgyrogroup of $G$, \eqref{eqn:
index formula} implies $\abs{G} = [G\colon\stab{x}]\abs{\stab{x}}$.
Because
$$[G\colon \stab{x}] = \abs{G/\stab{x}} = \abs{\orb{x}},$$ we obtain
$\abs{G} = \abs{\orb{x}}\abs{\stab{x}}$.
\end{proof}

\par Let $X$ be a $G$-set. The \textit{set of fixed points of $X$},
denoted by $\Fix{X}$, is defined as
\begin{equation}
\Fix{X} = \cset{x\in X}{a\cdot x = x \textrm{ for all }a\in G}.
\end{equation}
From \eqref{eqn: orbit} one finds that $x\in \Fix{X}$ if and only if
$\orb{x} = \set{x}$. The following theorem can be regarded as a
generalization of the class equation familiar from finite group
theory.

\begin{theorem}[The orbit decomposition theorem]\label{thm: Orbit decomposition theorem}
Let $G$ be a gyrogroup acting on a finite set $X$. Let $x_1$, $x_2,
\dots$, $x_n$ be representatives for the distinct nonsingleton
orbits in $X$. Then
\begin{equation}
\abs{X} = \abs{\Fix{X}}+\lsum{i=1}{n}[G\colon\stab{x_i}].
\end{equation}
\end{theorem}
\begin{proof}
Since $\cset{\orb{x}}{x\in X}$ forms a disjoint partition of $X$, it
follows from the orbit-stabilizer theorem that
\begin{align*}
\abs{X} &= \Abs{\lcup{x\in\Fix{X}}{}\orb{x}}+\Abs{\lcup{i=1}{n}\orb{x_i}}\\
{} &= \abs{\Fix{X}} + \lsum{i=1}{n}\abs{\orb{x_i}}\\
{} &= \abs{\Fix{X}} + \lsum{i=1}{n}[G\colon\stab{x_i}].\qedhere
\end{align*}
\end{proof}

\par If $G$ is a finite gyrogroup with trivial gyroautomorphisms, that
is, $G$ is a finite group, then $G$ acts on itself by
group-theoretic conjugation $$a\cdot x = (a\oplus x)\ominus a$$ for
all $a\in G$, $x\in G$. In this case, the set of fixed points of $G$
equals $Z(G)$, the group-theoretic center of $G$. Further, the
stabilizer of $x$ equals $C_G(x)$, the group-theoretic centralizer
of $x$ in $G$. From the orbit decomposition theorem, we recover the
{\it class equation} in finite group theory,
\begin{equation}
\abs{G} = \abs{Z(G)}+\lsum{i=1}{n}[G\colon C_G(x_i)],
\end{equation}
where $x_1$, $x_2, \ldots$, $x_n$ are representatives for the
distinct conjugacy classes of $G$ not contained in the center of
$G$.

\par As a consequence of the orbit-stabilizer theorem, we prove the
{\it Burnside lemma}, also known as the {\it Cauchy-Frobenius
lemma}, for finite gyrogroups.

\par Let $X$ be a $G$-set. Recall that for $x\in X$, the stabilizer of
$x$ in $G$ is the subgyrogroup
$$\stab{x} = \cset{a\in G}{a\cdot x = x}.$$
Dually, for $a\in G$, the set of elements of $X$ fixed by $a$ is
defined as
\begin{equation}
\fix{a} = \cset{x\in X}{a\cdot x = x}.
\end{equation}

\begin{theorem}[The Burnside lemma]\label{thm: Burnside's lemma}
Let $G$ be a finite gyrogroup and let $X$ be a finite $G$-set. The
number of distinct orbits in $X$ equals
$$\dfrac{1}{\abs{G}}\lsum{a\in G}{}\abs{\fix{a}}.$$
\end{theorem}
\begin{proof}
Define $Y = \cset{(a, x)\in G\times X}{a\cdot x = x}$. We count the
number of elements of $Y$ in two ways.

\par Note that for each $a\in G$, $(a, x)\in Y$ if and only if $a\cdot
x = x$ if and only if $x\in \fix{a}$. Hence, there are exactly
$\abs{\fix{a}}$ pairs in $Y$ with first coordinate $a$. It follows
that
\begin{equation}\label{eqn: Size of Y, First}
\abs{Y} = \lsum{a\in G}{}\abs{\fix{a}}.
\end{equation}
Dually, $\abs{Y} = \lsum{x\in X}{}\abs{\stab{x}}$. Assume that $X$
is partitioned into $n$ distinct orbits, namely $\orb{x_1},
\orb{x_2},\ldots, \orb{x_n}$. For each $x\in X$, $x$ belongs to
exactly one orbit, so
\begin{equation}\label{eqn: Size of Y, second}
\abs{Y} = \lsum{i=1}{n}\Bp{\lsum{x\in \orb{x_i}}{}\abs{\stab{x}}}.
\end{equation}
Note that $x$ belongs to $\orb{x_i}$ if and only if $\orb{x} =
\orb{x_i}$. By \eqref{eqn: orbit-stabilizer theorem},
$\abs{\stab{x}} = \abs{\stab{x_i}}$ for all $x\in\orb{x_i}$. Hence,
$$\lsum{x\in\orb{x_i}}{}\abs{\stab{x}} =
\lsum{x\in\orb{x_i}}{}\abs{\stab{x_i}} =
\abs{\orb{x_i}}\abs{\stab{x_i}} = \abs{G}.$$ By \eqref{eqn: Size of
Y, second},
\begin{equation}\label{eqn: Size of Y, third}
\abs{Y} = \lsum{i=1}{n}\abs{G} = n\abs{G}.
\end{equation}
Equating \eqref{eqn: Size of Y, First} and \eqref{eqn: Size of Y,
third} completes the proof.
\end{proof}

\par The following results lead to a deep understanding of
\textit{transitive} gyrogroup actions, which will be studied in
detail in Sections \ref{subsec: type of action} and \ref{sec: left
gyroaddition}.

\begin{theorem}\label{thm: stab (a.z) = conjugate of stab z}
Let $X$ be a $G$-set. For all $a\in G$, $x\in X$,
$$\stab{(a\cdot x)} = \cset{(a\oplus c)\boxminus a}{c\in\stab{x}}.$$
\end{theorem}
\begin{proof}
For $a$, $b\in G$, $x\in X$, direct computation shows that $b\in
\stab{(a\cdot x)}$ if and only if $\ominus a\oplus (b\oplus
a)\in\stab{x}$. The theorem is an application of the right
cancellation laws I and II.
\end{proof}

\par Because of the absence of associativity in gyrogroups, the
expression \mbox{$a\oplus b\ominus a$} is ambiguous. Hence, the
conjugate of $b$ by $a$ cannot be defined as $a\oplus b\ominus a$,
as in group theory. We formulate an appropriate notion of conjugate
elements in a gyrogroup, which is motivated by Theorem \ref{thm:
stab (a.z) = conjugate of stab z}.

\begin{definition}[Conjugates]\label{def: conjugation}
The element $(a\oplus b)\boxminus a$ is called the \textit{conjugate
of $b$ by $a$}. For a given subset $B$ of $G$, the \textit{conjugate
of $B$ by a}, denoted by $\conj{a}{B}$, is defined by
\begin{equation}
\conj{a}{B} = \cset{(a\oplus b)\boxminus a}{b\in B}.
\end{equation}
An element $a$ is {\it conjugate to} an element $b$ if $a$ is the
conjugate of $b$ by some element of $G$. A subset $A$ is {\it
conjugate to} a subset $B$ if $A$ is the conjugate of $B$ by some
element of $G$.
\end{definition}

\par If $G$ is a gyrogroup with trivial gyroautomorphisms, then
$a\boxplus b = a\oplus b$ for all $a$, $b\in G$ and the gyrogroup
operation $\oplus$ is associative. In this case, the notion of
conjugation in Definition \ref{def: conjugation} reduces to that of
group-theoretic conjugation.

\begin{theorem}\label{thm: related elements have conjugate stabilizer}
Let $X$ be a $G$-set. For all $x$, $y\in X$, if $x\sim y$, then
$\stab{x}$ and $\stab{y}$ are conjugate to each other.
\end{theorem}
\begin{proof}
Suppose that $x\sim y$. Then $y = a\cdot x$ for some $a\in G$. By
Theorem \ref{thm: stab (a.z) = conjugate of stab z},
$$\stab{y} = \stab{(a\cdot x)} = \conj{a}{\stab{x}}.$$
Similarly, $x = (\ominus a)\cdot y$ implies $\stab{x} =
\conj{\ominus a}{\stab{y}}$.
\end{proof}

\subsection{Invariant subsets}
\par Let $G$ be a gyrogroup acting on $X$. For a given subset $Y$ of
$X$, define
\begin{equation}
GY = \cset{a\cdot y}{a\in G,\, y\in Y}.
\end{equation}
We say that $Y$ is an \textit{invariant} subset of $X$ if $GY
\subseteq Y$. Clearly, $Y$ is an invariant subset of $X$ if and only
if $GY = Y$. For each $x\in X$, $\orb{x}$ is indeed an invariant
subset of $X$. More generally, if $Y$ is a nonempty subset of $X$,
then $GY$ is an invariant subset of $X$. Invariant subsets play the
role of substructures of a $G$-set, as the following proposition
indicates.

\begin{proposition}\label{prop: action on invariant subset}
Let $X$ be a $G$-set. If $Y$ is an invariant subset of $X$, then $G$
acts on $Y$ by the restriction of the action of $G$ to $Y$.
\end{proposition}
\begin{proof}
Because $GY = Y$, $a\cdot y$ belongs to $Y$ for all $a\in G$ and
$y\in Y$. That the axioms of a gyrogroup action are satisfied is
immediate.
\end{proof}

\subsection{Types of actions}\label{subsec: type of action}
\par Following terminology in the theory of group actions, we present
several types of gyrogroup actions. If $G$ is a gyrogroup with
gyroautomorphisms being the identity automorphism, the terminology
corresponds to its group counterpart.

\begin{definition}[Faithful actions]\label{def: faithful action}
A gyrogroup action of $G$ on $X$ is \textit{faithful} if the
gyrogroup homomorphism afforded by the action is injective.
\end{definition}

\begin{theorem}\label{thm: faithful action from any action}
Let $X$ be a $G$-set and let $\pi\colon G\to\sym{X}$ be the
associated permutation representation. The quotient gyrogroup
$G/\ker{\pi}$ acts faithfully on $X$ by
\begin{equation}\label{eqn: quotient action on X}
(a\oplus \ker{\pi})\cdot x = a\cdot x
\end{equation}
for all $a\in G$, $x\in X$.
\end{theorem}
\begin{proof}
Set $K = \ker{\pi}$. By Lemma \ref{lem: Equivalence of equal cosets
of stabilizer}, \eqref{eqn: quotient action on X} is well defined.
Since $G/K$ admits the quotient gyrogroup structure, we have
$$(a\oplus K)\cdot ((b\oplus K)\cdot x) = ((a\oplus K)\oplus(b\oplus
K))\cdot x$$ for all $a, b\in G, x\in X$. This proves that
\eqref{eqn: quotient action on X} defines an action of $G/K$ on $X$.

Let $\dot\vphi$ be the gyrogroup homomorphism afforded by the action
\eqref{eqn: quotient action on X}. Suppose that $a\oplus K\in G/K$
is such that $\dot\vphi(a\oplus K) = \id{X}$. Hence, $(a\oplus
K)\cdot x = x$ for all $x\in X$. By \eqref{eqn: quotient action on
X}, $a\cdot x = x$ for all $x\in X$. Hence, $a\in K$ and so $a\oplus
K = 0\oplus K$. This proves that $\dot\vphi$ is injective and hence
the action \eqref{eqn: quotient action on X} is faithful.
\end{proof}

\begin{definition}[Transitive actions]\label{def: transitive action}
A gyrogroup action of $G$ on $X$ is \textit{transitive} if for each
pair of points $x$ and $y$ in $X$, there is an element $a$ of $G$
such that $a\cdot x = y$.
\end{definition}

\begin{definition}[Fixed-point-free actions]\label{def: free action}
An action of a gyrogroup $G$ on a set $X$ is \textit{fixed point
free} (or simply {\it free}) if $\stab{x} = \set{0}$ for all $x\in
X$.
\end{definition}

\begin{proposition}\label{prop: characterization of free action}
Let $X$ be a $G$-set. The following are equivalent:
\begin{enumerate}
    \item\label{item: free action} The action of $G$ on $X$ is fixed point free.
    \item\label{item: ax = bx --> a = b} For all $a$, $b\in G$, $x\in X$, $a\cdot x = b\cdot x$ implies $a = b$.
    \item\label{item: ax = x --> a = 0} For all $a \in G$, $x\in X$, $a\cdot x = x$ implies $a = 0$.
\end{enumerate}
\end{proposition}
\begin{proof}
The proof of the proposition is straightforward, using Definition
\ref{def: free action} and the defining properties of a gyrogroup
action.
\end{proof}

\begin{definition}[Semiregular actions]\label{def: semiregular action}
An action of a gyrogroup $G$ on a set $X$ is \textit{semiregular} if
there exists a point $z$ in $X$ such that $\stab{z} = \set{0}$.
\end{definition}

\begin{proposition}\label{prop: fixed point free --> semiregular}
Every fixed-point-free action is semiregular. Every semiregular
\mbox{action} is faithful.
\end{proposition}
\begin{proof}
The first statement is immediate from Definitions \ref{def: free
action} and \ref{def: semiregular action}. The second statement
follows directly from Proposition \ref{prop: Stabilizer
subgyrogroup} \eqref{item: Kerner of varphidot as intersection} and
Definition \ref{def: semiregular action}.
\end{proof}

\begin{theorem}\label{thm: equivalence of fixed point free and semiregular in transitive action}
A transitive action is fixed point free if and only if it is
semiregular.
\end{theorem}
\begin{proof}
The forward implication was proved in Proposition \ref{prop: fixed
point free --> semiregular}. Conversely, suppose that the action is
transitive and that $\stab{z} = \set{0}$ for some $z\in X$. Let
$x\in X$. Since the action is transitive, $x\sim z$. By Theorem
\ref{thm: related elements have conjugate stabilizer}, $\stab{x}$
and $\stab{z}$ are conjugate to each other. Since $\stab{z} =
\set{0}$, it follows that $\stab{x} = \set{0}$.
\end{proof}

\begin{definition}[Sharply transitive actions]\label{def: sharply transitive}
A gyrogroup action of $G$ on $X$ is \textit{sharply transitive} or
\textit{regular} if for each pair of points $x$ and $y$ in $X$,
there exists a unique element $a$ of $G$ such that $a\cdot x = y$.
\end{definition}

\begin{theorem}\label{thm: characterization of sharply transitive general case}
An action of a gyrogroup $G$ on a set $X$ is sharply transitive if
and only if it is transitive and fixed point free.
\end{theorem}
\begin{proof}
($\Rightarrow$) It is clear that $G$ acts transitively on $X$. Let
$x\in X$ and let $a\in\stab{x}$. Then $a\cdot x = x$. Since $0$ is
the unique element of $G$ such that $0\cdot x = x$, it follows that
$a = 0$. This proves $\stab{x}\subseteq\set{0}$ and so equality
holds.

($\Leftarrow$) Let $x, y\in X$. Since the action is transitive,
there is an element $a\in G$ for which $a\cdot x = y$. Suppose that
$b\in G$ is such that $b\cdot x = y$. Hence, $a\cdot x = b\cdot x$,
which implies $a = b$ by Proposition \ref{prop: characterization of
free action}.
\end{proof}

\begin{theorem}\label{thm: characterization of sharply transitive II}
An action of a gyrogroup $G$ on a set $X$ is sharply transitive if
and only if it is transitive and semiregular.
\end{theorem}
\begin{proof}
The forward implication follows from Theorem \ref{thm:
characterization of sharply transitive general case} and Proposition
\ref{prop: fixed point free --> semiregular}. The converse follows
from Theorems \ref{thm: equivalence of fixed point free and
semiregular in transitive action} and \ref{thm: characterization of
sharply transitive general case}.
\end{proof}

\par Theorem \ref{thm: equivalence of fixed point free and semiregular
in transitive action} asserts that on the class of transitive
gyrogroup actions, the notions of fixed point free actions and
semiregular actions are equivalent. The proof of Theorem \ref{thm:
equivalence of fixed point free and semiregular in transitive
action} is a good example to see how the internal structure of a
$G$-set interacts with the gyrogroup structure of $G$. Next, we
study transitive actions in detail. In particular, we give a
necessary and sufficient condition on a subgyrogroup $H$ of $G$ so
that $G$ acts transitively on the coset space $G/H$ in a natural
way.

\section{Actions by left gyroaddition}\label{sec: left gyroaddition}
\par As in the theory of group actions, the study of permutation
representations of a gyrogroup $G$ reduces to the study of
transitive permutation representations of $G$. The transitive
permutation representations of $G$ in turn are determined by the
structure of subgyrogroups of $G$. From this point of view, the
transitive permutation representations of gyrogroups play the role
of indecomposable representations in the theory of gyrogroup
actions. In this section, we develop the elementary theory of
transitive gyro-group actions, following the study of group actions
by Aschbacher \cite{MA2000FGT}.

\par Recall that a group $\Gam$ acts on itself by {\it left
multiplication} $g\cdot x = gx$ for all $g\in \Gam$ and $x\in \Gam$.
In contrast, a gyrogroup $G$ does not in general act on itself by
left gyroaddition, as Proposition \ref{prop: G acts on itselft by
left gyroaddition} indicates. Therefore, we will determine a
necessary and sufficient condition for a gyrogroup $G$ to act on its
coset space $G/H$ by left gyroaddition.

\begin{proposition}\label{prop: G acts on itselft by left gyroaddition}
A gyrogroup $G$ acts on itself by
\begin{equation}\label{eqn: left gyroaddition}
a\cdot x = a\oplus x,\qquad a\in G,\, x\in G,
\end{equation}
if and only if $\gyr{a, b}{} = \id{G}$ for all $a$, $b\in G$.
\end{proposition}
\begin{proof}
Suppose that \eqref{eqn: left gyroaddition} defines an action of $G$
on itself. Let $a$, $b\in G$. For $x\in G$, we have $(a\oplus
b)\oplus x = a\cdot(b\cdot x) = (a\oplus b)\oplus\gyr{a, b}{x}$.
Hence, $\gyr{a, b}{x} = x$ by the left cancellation law. Since $x$
is arbitrary, $\gyr{a, b}{} = \id{G}$. If $\gyr{a, b}{} = \id{G}$
for all $a, b\in G$, then $(G, \oplus)$ is a group and the converse
holds trivially.
\end{proof}

\begin{lemma}\label{lem: Equivalence of gyr is identity on coset space}
Let $H$ be a subgyrogroup of $G$. Then $\gyr{a, b}{(x\oplus
H)}\subseteq x\oplus H$ for all $a, b, x\in G$ if and only if
\begin{enumerate}
    \item\label{item: invariance under gyromap} $\gyr{a, b}{(H)}\subseteq H$ for all $a$, $b\in G$, and
    \item\label{item: -x+gyr[a, b]x in H} $\ominus x \oplus\gyr{a, b}{x}\in H$ for all $a$, $b$, $x\in G$.
\end{enumerate}
\end{lemma}
\begin{proof}
($\Rightarrow$) By Proposition 6 of \cite{TSKW2015ITG}, $\gyr{a,
b}{(x\oplus H)} = x\oplus H$ for all $a$, $b$, $x\in G$. Setting $x
= 0$, we obtain item \eqref{item: invariance under gyromap}. From
item \eqref{item: invariance under gyromap}, we obtain $H\leqslant_L
G$. Since $\gyr{a, b}{}$ preserves $\oplus$ and $\gyr{a, b}{(H)} =
H$, we have $\gyr{a, b}{(x\oplus H)} = (\gyr{a, b}{x})\oplus H$.
Since $x\oplus H = \gyr{a, b}{(x\oplus H)} = (\gyr{a, b}{x})\oplus
H$, it follows that $\ominus x\oplus\gyr{a, b}{x}\in H$, which
proves item \eqref{item: -x+gyr[a, b]x in H}.

($\Leftarrow$) By condition \eqref{item: invariance under gyromap},
$H\leqslant_L G$. Hence, condition \eqref{item: -x+gyr[a, b]x in H}
is equivalent to saying that $x\sim_H\gyr{a, b}{x}$, which implies
$x\oplus H = (\gyr{a, b}{x})\oplus H = \gyr{a, b}{(x\oplus H)}$.
\end{proof}

\par The following theorem gives a necessary and sufficient condition for
a gyrogroup $G$ to act on its coset space $G/H$ in a natural way.

\begin{theorem}\label{thm: Left action on coset space}
Let $H$ be a subgyrogroup of $G$. Then $\gyr{a, b}{(x\oplus
H)}\subseteq x\oplus H$ for all $a$, $b$, $x\in G$ if and only if
$G$ acts on the coset space $G/H$ by
\begin{equation}\label{eqn: left gyroaddition action}
a\cdot (x\oplus H) = (a\oplus x)\oplus H
\end{equation}
for all $a\in G$, $x\oplus H\in G/H$.
\end{theorem}
\begin{proof}
($\Rightarrow$) By item \eqref{item: invariance under gyromap} of
Lemma \ref{lem: Equivalence of gyr is identity on coset space},
$H\leqslant_L G$. Hence, $G/H$ forms a disjoint partition of $G$,
and $x\oplus H = y\oplus H$ if and only if $\ominus x\oplus y \in H$
if and only if $\ominus y\oplus x\in H$ for all $x$, $y\in G$. Using
the left gyroassociative law together with Lemma \ref{lem:
Equivalence of gyr is identity on coset space}, one can prove that
\eqref{eqn: left gyroaddition action} is well defined.

Clearly, $0\cdot (x\oplus H) = x\oplus H$ for all $x\in G$. Let $a$,
$b$, $x$ be arbitrary elements of $G$. On one hand, we have $$a\cdot
(b\cdot (x\oplus H)) = (a\oplus(b\oplus x))\oplus H = ((a\oplus
b)\oplus \gyr{a, b}{x})\oplus H.$$ On the other hand, we have
$(a\oplus b)\cdot (x\oplus H) = ((a\oplus b)\oplus x)\oplus H$. From
the gyrator identity \eqref{eqn: gyrator identity} and the left
cancellation law, we obtain
$$
\gyr{a\oplus b, x}{(\ominus x\oplus\gyr{a, b}{x})} =
\ominus((a\oplus b)\oplus x)\oplus((a\oplus b)\oplus \gyr{a, b}{x}).
$$
By item \eqref{item: -x+gyr[a, b]x in H} of Lemma \ref{lem:
Equivalence of gyr is identity on coset space}, $\ominus x\oplus
\gyr{a, b}{x}\in H$ and by item \eqref{item: invariance under
gyromap} of the same lemma, $\gyr{a\oplus b, x}{(\ominus
x\oplus\gyr{a, b}{x})}\in H$. Hence, $\ominus((a\oplus b)\oplus
x)\oplus((a\oplus b)\oplus \gyr{a, b}{x})$ is in $H$. It follows
that $((a\oplus b)\oplus \gyr{a, b}{x})\oplus H = ((a\oplus b)\oplus
x)\oplus H$, which implies $a\cdot (b\cdot (x\oplus H)) = (a\oplus
b)\cdot (x\oplus H)$.

($\Leftarrow$) Suppose that \eqref{eqn: left gyroaddition action}
defines an action of $G$ on $G/H$ and let $a$, $b$, $x\in G$. From
the axioms of a gyrogroup action, we have
\begin{equation}\label{eqn: in proof caraterization of left regular action I}
\ominus(a\oplus b)\cdot ((a\oplus b)\cdot (x\oplus H)) = x\oplus H
\end{equation}
and
\begin{equation}\label{eqn: in proof caraterization of left regular action II}
(a\oplus b)\cdot (x\oplus H) = a\cdot(b\cdot (x\oplus H)) =
(a\oplus(b\oplus x))\oplus H.
\end{equation}
From \eqref{eqn: in proof caraterization of left regular action II}
and the gyrator identity, we have
\begin{align}\label{eqn: in proof caraterization of left regular action III}
\begin{split}
\ominus(a\oplus b)\cdot((a\oplus b)\cdot (x\oplus H)) &=
\ominus(a\oplus b)\cdot((a\oplus (b\oplus x))\oplus H)\\
{} &= \ominus(a\oplus b)\oplus(a\oplus(b\oplus x))\oplus H\\
{} &= (\gyr{a, b}{x})\oplus H.
\end{split}
\end{align}
Equating \eqref{eqn: in proof caraterization of left regular action
I} and \eqref{eqn: in proof caraterization of left regular action
III} gives $(\gyr{a, b}{x})\oplus H = x\oplus H$. In the special
case when $h\in H$, we have $(\gyr{a, b}{h})\oplus H = h\oplus H =
H$, which implies $\gyr{a, b}{h}\in H$. This proves $\gyr{a,
b}{(H)}\subseteq H$ for all $a$, $b\in G$. As in the proof of Lemma
\ref{lem: Equivalence of gyr is identity on coset space}, we have
$\gyr{a, b}{(x\oplus H)} = (\gyr{a, b}{x})\oplus H$ and so $\gyr{a,
b}{(x\oplus H)} = x\oplus H$.
\end{proof}

\par Theorem \ref{thm: Left action on coset space} suggests the
following definition.

\begin{definition}[Left-gyroaddition actions]\label{def: left gyroaddition action}
Let $H$ be a subgyrogroup of $G$. The gyro-group $G$ acts on the
coset space $G/H$ by \textit{left gyroaddition} if \eqref{eqn: left
gyroaddition action} defines a gyrogroup action of $G$ on $G/H$.
\end{definition}

\begin{theorem}\label{thm: necessary and sufficient condition for left gyroaddition action}
Let $H$ be a subgyrogroup of $G$. Then $G$ acts on the coset space
$G/H$ by left gyroaddition if and only if $$\gyr{a, b}{(x\oplus
H)}\subseteq x\oplus H$$ for all $a$, $b\in G$, $x\oplus H\in G/H$.
\end{theorem}
\begin{proof}
The theorem follows directly from Theorem \ref{thm: Left action on
coset space} and Definition \ref{def: left gyroaddition action}.
\end{proof}

\begin{theorem}\label{thm: Orbit and stabilizer of left action}
Let $H$ be a subgyrogroup of $G$ such that $$\gyr{a, b}{(x\oplus
H)}\subseteq x\oplus H$$ for all $a, b, x\in G$. Then $G$ acts
transitively on $G/H$ by left gyroaddition. The stabilizer of a
point $x\oplus H$ is the conjugate of $H$ by $x$, that is,
\begin{equation}
\stab{(x\oplus H)} = \cset{(x\oplus h)\boxminus x}{h\in H}.
\end{equation}
\end{theorem}
\begin{proof}
By Theorem \ref{thm: necessary and sufficient condition for left
gyroaddition action}, $G$ acts on $G/H$ by left gyroaddition. Let
$x\oplus H$, $y\oplus H\in G/H$. Set $a = y\boxminus x$. By the
right cancellation law II, $a\cdot (x\oplus H) = y\oplus H$. This
proves that $G$ acts transitively on $G/H$. Let $x\in G$. Then
$x\oplus H = x\cdot(0\oplus H)$ and by Theorem \ref{thm: related
elements have conjugate stabilizer}, $\stab{(x\oplus H)} =
\conj{x}{0\oplus H} = \conj{x}{H}$.
\end{proof}

\begin{theorem}\label{thm: left gyroaddition is not semiregular}
If $H\ne\set{0}$ and $\gyr{a, b}{(x\oplus H)}\subseteq x\oplus H$
for all $a$, $b$, $x\in G$, then the action by left gyroaddition is
not semiregular.
\end{theorem}
\begin{proof}
By Theorem \ref{thm: Orbit and stabilizer of left action}, the
stabilizer of $x\oplus H$ equals $\conj{x}{H}$. By Theorem \ref{thm:
cancellation law in gyrogroup}, $(x\oplus h)\boxminus x = 0$ if and
only if $h = 0$. Hence, $\stab{(x\oplus H)}\ne\set{0}$ because
$H\ne\set{0}$.
\end{proof}

\par The following theorem gives an arithmetic necessary condition for a
finite gyro-group to act on its coset space by left gyroaddition.

\begin{theorem}\label{thm: necessary condition for having left gyroaddition}
Let $H$ be a subgyrogroup of a finite gyrogroup $G$. If $G$ acts on
$G/H$ by left gyroaddition, then the index formula holds, $\abs{G} =
[G\colon H]\abs{H}$.
\end{theorem}
\begin{proof}
By Theorem \ref{thm: Orbit and stabilizer of left action}, $G$ acts
transitively on $G/H$. Hence, $\abs{\orb{(0\oplus H)}} = [G\colon
H]$. Furthermore, $\stab{(0\oplus H)} = H$. By the orbit-stabilizer
theorem,
$$\abs{G} = \abs{\orb{(0\oplus H)}}\abs{\stab{(0\oplus H)}} = [G\colon H]\abs{H},$$
which was to be proved.
\end{proof}

\par The remainder of this section is devoted to a complete description
of transitive actions of a gyrogroup $G$ on a set $X$. We will prove
that any transitive action of $G$ on $X$ is equivalent to an action
of $G$ on the coset space of a stabilizer subgyrogroup of $G$ by
left gyroaddition. We also state the {\it fundamental isomorphism
theorem for $G$-sets}. We end this section by proving that any
permutation representation of $G$ is uniquely determined by its
orbits, up to rearrangement.

\begin{theorem}\label{thm: action induces left gyroaddition on G/H}
Let $X$ be a $G$-set and let $\dot\vphi\colon G\to\sym{X}$ be the
associated permutation representation. If $H$ is a subgyrogroup of
$G$ containing $\ker{\dot\vphi}$ and if $\gyr{a, b}{(H)}\subseteq H$
for all $a$, $b\in G$, then $G$ acts transitively on $G/H$ by left
gyroaddition.
\end{theorem}
\begin{proof}
In light of Theorem \ref{thm: Orbit and stabilizer of left action},
we will prove that $\gyr{a, b}{(x\oplus H)} \subseteq x\oplus H$ for
all $a, b, x\in G$. Note that $\dot\vphi(\gyr{a, b}{c}) =
\gyr{\dot\vphi(a), \dot\vphi(b)}{\dot\vphi(c)} = \dot\vphi(c)$ for
all $a$, $b$, $c\in G$ since the gyroautomorphisms of $\sym{X}$ are
the identity automorphism.

Let $a$, $b$, $x\in G$. Then $\dot\vphi(\gyr{a, b}{x}) =
\dot\vphi(x)$. Hence, $\ominus x\oplus\gyr{a, b}{x}\in
\ker{\dot\vphi}$. Since $\ker{\dot\vphi}\subseteq H$, we have
$\ominus x\oplus\gyr{a, b}{x}\in H$. By assumption, $\gyr{a,
b}{(H)}\subseteq H$ for all $a, b\in G$. By Lemma \ref{lem:
Equivalence of gyr is identity on coset space}, $\gyr{a, b}{(x\oplus
H)} \subseteq x\oplus H$, which completes the proof.
\end{proof}

\begin{corollary}
Let $X$ be a $G$-set and let $\dot\vphi$ be the associated
permutation representation. Then $G$ acts transitively on
$G/\ker{\dot\vphi}$ by left gyroaddition.
\end{corollary}
\begin{proof}
It is proved in \cite{TSKW2015ITG} that $\gyr{a,
b}{(\ker{\dot\vphi})}\subseteq \ker{\dot\vphi}$ for all $a$, $b\in
G$. Hence, \mbox{Theorem} \ref{thm: action induces left gyroaddition
on G/H} applies to $\ker{\dot\vphi}$.
\end{proof}

\begin{corollary}\label{cor: G acts on its stabilizer}
Let $X$ be a $G$-set. For each $x\in X$, $G$ acts transitively on
$G/\stab{x}$ by left gyroaddition.
\end{corollary}
\begin{proof}
By Proposition \ref{prop: Stabilizer subgyrogroup} \eqref{item:
Kerner of varphidot as intersection},
$\ker{\dot\vphi}\subseteq\stab{x}$. By Proposition \ref{prop:
Stabilizer invariant under gyroautomorphism}, $$\gyr{a,
b}{(\stab{x})}\subseteq\stab{x}$$ for all $a$, $b\in G$. Hence,
Theorem \ref{thm: action induces left gyroaddition on G/H} applies
to $\stab{x}$.
\end{proof}

\begin{definition}[$G$-maps and equivalences]\label{def: G-map}
Let $X$ and $Y$ be $G$-sets. A map $\Phi\colon X\to Y$ is a
\textit{$G$-map} if $$\Phi(a\cdot x) = a\cdot\Phi(x)$$ for all $a\in
G$, $x\in X$. A bijective $G$-map from $X$ to $Y$ is called an
\textit{equivalence}. If there exists an equivalence from $X$ to
$Y$, $X$ and $Y$ are said to be \textit{equivalent}, denoted by
$X\equiv Y$.
\end{definition}

\par Intuitively, if $X$ and $Y$ are equivalent $G$-sets, then $X$ and
$Y$ are algebraically identical except that the elements and the
actions may be written differently in $X$ and $Y$.

\begin{theorem}[The fundamental isomorphism theorem]
Let $X$ be a $G$-set and let $z\in X$. The map $\Phi$ defined by
\begin{equation}
\Phi(a\oplus \stab{z}) = a\cdot z, \qquad a\in G,
\end{equation}
is an equivalence from $G/\stab{z}$ to $\orb{z}$. In particular,
$G/\stab{z}\equiv \orb{z}$.
\end{theorem}
\begin{proof}
By Corollary \ref{cor: G acts on its stabilizer}, $G$ acts on
$G/\stab{z}$ by left gyroaddition. By \mbox{Proposition} \ref{prop:
action on invariant subset}, $G$ acts on $\orb{z}$ by the
restriction of the action of $G$ to $\orb{z}$. Hence, $G/\stab{z}$
and $\orb{z}$ are $G$-sets. Note that $\Phi$ is indeed the
\mbox{inverse} map of the map $\theta$ given in the proof of Theorem
\ref{thm: Orbit-stabilizer theorem} with $x = z$. Hence, $\Phi$ is a
bijection from $G/\stab{z}$ to $\orb{z}$.

To verify that $\Phi$ is an equivalence from $G/\stab{z}$ to
$\orb{z}$, we compute
\begin{align*}
\Phi(a\cdot (x\oplus \stab{z})) &= \Phi((a\oplus x)\oplus \stab{z})\\
{} &= (a\oplus x)\cdot z\\
{} &= a\cdot(x\cdot z)\\
{} &= a\cdot \Phi(x\oplus \stab{z})
\end{align*}
for all $a\in G$, $x\in G$. Hence, $G/\stab{z}\equiv\orb{z}$.
\end{proof}

\begin{corollary}\label{cor: consequence of the 1st isomorphism theorem}
If $G$ acts transitively on $X$, then $G/\stab{z}\equiv X$ for all
$z\in X$.
\end{corollary}
\begin{proof}
Since $\orb{z} = X$, it follows from the fundamental isomorphism
\mbox{theorem} for $G$-sets that $G/\stab{z}\equiv X$.
\end{proof}

\begin{lemma}\label{lem: stabilizer of two equivalent sets}
If $X$ and $Y$ are equivalent $G$-sets via an equivalence
\mbox{$\Phi\colon X\to Y$}, then
$$\stab{x} = \stab{\Phi(x)}$$
for all $x\in X$.
\end{lemma}
\begin{proof}
The proof of the lemma is straightforward, using the fact that
$\Phi$ is an equivalence.
\end{proof}

\begin{theorem}
Suppose that $G$ acts transitively on $X$ and $Y$ and let $x\in X$,
$y\in Y$. Then $X\equiv Y$ if and only if $\stab{x}$ is conjugate to
$\stab{y}$ in the sense of Definition \ref{def: conjugation}.
\end{theorem}
\begin{proof}
($\Rightarrow$) Let $\Phi\colon X\to Y$ be an equivalence. By Lemma
\ref{lem: stabilizer of two equivalent sets}, $\stab{x} =
\stab{\Phi(x)}$. Since $G$ acts transitively on $Y$, $y\sim
\Phi(x)$. By Theorem \ref{thm: related elements have conjugate
stabilizer}, $\stab{\Phi(x)}$ is conjugate to $\stab{y}$.

($\Leftarrow$) Suppose that $\stab{x} = \conj{a}{\stab{y}}$ for some
$a\in G$. By Theorem \ref{thm: stab (a.z) = conjugate of stab z},
$$\stab{(a\cdot y)} =\conj{a}{\stab{y}}.$$ Thus, $\stab{x} =
\stab{(a\cdot y)}$. By Corollary \ref{cor: consequence of the 1st
isomorphism theorem}, $G/\stab{x}\equiv X$ and $G/\stab{(a\cdot
y)}\equiv Y$. Hence, $X\equiv Y$.
\end{proof}

\par Let $X$ be an arbitrary $G$-set. The \textit{transitive
components of $X$}, denoted by $\tranc{X}$, is defined to be the
collection of all orbits of $X$ in $G$, that is,
\begin{equation}
\tranc{X} = \cset{\orb{x}}{x\in X}.
\end{equation}
By Theorem \ref{thm: Equivalence relation from action} and
\eqref{eqn: orbit}, $\tranc{X}$ forms a disjoint partition of $X$.
For each $x\in X$, since $\orb{x}$ is an invariant subset of $X$,
Proposition \ref{prop: action on invariant subset} asserts that $G$
acts transitively on $\orb{x}$. The following theorem shows that any
$G$-set is uniquely determined by its transitive components, up to
rearrangement.

\begin{theorem}\label{thm: action is uniquely determined by its transitive component}
Let $X$ and $Y$ be $G$-sets. Then $X\equiv Y$ if and only if there
is a bijection $\bet\colon \tranc{X}\to \tranc{Y}$ such that
$\orb{x}\equiv \bet(\orb{x})$ for all $x\in X$.
\end{theorem}
\begin{proof}
($\Rightarrow$) Suppose that $\Phi\colon X\to Y$ is an equivalence.
Since $\tranc{Y}$ is a disjoint partition of $Y$, one can define the
map $\bet$ by
$$\bet(\orb{x}) = \orb{\Phi(x)},\qquad x\in X.$$
The restriction of $\Phi$ to $\orb{x}$ acts as an equivalence from
$\orb{x}$ to $\bet(\orb{x})$.

($\Leftarrow$) Let $\cset{x_i}{i\in I}$ be the set of
representatives for the distinct orbits in $X$ with the property
that $i\ne j$ implies $\orb{x_i}\ne\orb{x_j}$. By assumption, for
each $i\in I$, there exists an equivalence $\Phi_i$ from $\orb{x_i}$
to $\bet(\orb{x_i})$. For each $z\in X$, there is a unique index
$i\in I$ such that  $z\in\orb{x_i}$ by the defining property of
$\cset{x_i}{i\in I}$. Hence, we can define the map $\Phi$ by the
condition
$$z\in\orb{x_i}\quad\textrm{implies}\quad\Phi(z) = \Phi_i(z)$$
for all $z\in X$. It can be shown that $\Phi$ is an equivalence from
$X$ to $Y$.
\end{proof}

\par Theorem \ref{thm: action is uniquely determined by its transitive
component} indicates that the study of permutation representations
of a gyrogroup $G$ is reduced to the study of transitive permutation
representations of $G$. Corollary \ref{cor: consequence of the 1st
isomorphism theorem} indicates that the transitive permutation
representations of $G$ in turn are determined by the structure of
subgyrogroups of $G$ itself. Hence, knowing the subgyrogroups of $G$
amounts to knowing the transitive actions of $G$ on any nonempty set
$X$.

\section{Examples}\label{sec: example}
\par In this section, we provide concrete examples of transitive
gyrogroup actions. Speci-fically, we show that the Einstein and
M\"{o}bius gyrogroups on the open unit ball of $n$-dimensional
Euclidean space $\R^n$ give rise to transitive gyrogroup actions.
M\"{o}bius and Einstein gyrogroups themselves are of great
importance in gyrogroup theory as they provide concrete models for
an abstract gyrogroup. See for instance \cite{TSKW2015EGB,
MFGR2011MGC, ODES2013CMG, AU2008FMG, SKJL2013UBL, AU2007EVA,
AU2008AHG, AU2001BEA}.

\subsection{Gyrogroup of order 15}
\par In \cite{TS2015TAG}, the gyrocommutative gyrogroup $G_{15} =
\set{0, 1, 2,\dots, 14}$ is given. The gyro-automorphisms of
$G_{15}$ form the cyclic subgroup $\set{I,\, A,\, \dots,\, A^4}$ of
the symmetric group $\sym{G_{15}}$ generated by the gyroautomorphism
$A$, where $A$ has cycle decomposition given by
\begin{equation}\label{eqn: gyroautomorphism A}
A = (1\nulll 7 \nulll 5 \nulll 10 \nulll 6)(2 \nulll 3 \nulll 8
\nulll 11 \nulll 14).
\end{equation}

\par In $G_{15}$, $H = \set{0, 4, 9, 12, 13}$ forms an L-subgyrogroup
of $G_{15}$ and the coset space $G_{15}/H$ consists of three
distinct left cosets
\begin{align*}
0\oplus H &= \set{0, 4, 9, 12, 13},\\
1\oplus H &= \set{1, 5, 6, 7, 10},\\
2\oplus H &= \set{2, 3, 8, 11, 14}.
\end{align*}
We have by inspection that $\gyr{a, b}{(x\oplus H)} = x\oplus H$ for
all $a, b, x\in G_{15}$. By Theorem \ref{thm: Orbit and stabilizer
of left action}, $G_{15}$ acts transitively on $G_{15}/H$ by left
gyroaddition.

\subsection{Gyrogroup of permutations and Einstein and M\"{o}bius gyrogroups}
\par Let $G$ be a gyrogroup. For each $a\in G$, the \textit{left
gyrotranslation by $a$}, $L_a$, is defined by
$$L_a(x) = a\oplus x,\qquad x\in G.$$ By the left
cancellation law, $L_a$ defines a permutation of $G$ for all $a\in
G$. Set $$\hat{G} = \cset{L_a}{a\in G}$$ and let $\symz{G}$ denote
the group of permutations of $G$ leaving the \mbox{gyrogroup}
identity fixed, that is,
$$\symz{G} = \cset{\rho\in \sym{G}}{\rho(0)= 0}.$$
It is proved in \cite{TSKW2015ITG} that any permutation $\sig$ of
$G$ can be written uniquely as $\sig = L_a\circ \rho$, where $a\in
G$ and $\rho\in\symz{G}$, and that $\hat{G}$ with operation defined
by
\begin{equation}
L_a\oplus L_b = L_{a\oplus b},\qquad a,\, b\in G,
\end{equation}
forms a gyrogroup isomorphic to $G$ via the gyrogroup isomorphism
$a\mapsto L_a$.

\par Let $\sig$ and $\tau$ be arbitrary permutations of $G$. Suppose
that $\sig$ and $\tau$ have factorizations $\sig = L_a\circ\alp$ and
$\tau = L_b\circ \bet$, where $a$, $b\in G$ and $\alp$,
$\bet\in\symz{G}$. By Theorem 12 of \cite{TSKW2015ITG}, $\sym{G}$
with operation defined by
\begin{equation}\label{eqn: gyroaddition in Sym (G)}
\sig\oplus\tau = L_{a\oplus b}\circ (\alp\circ\bet)
\end{equation}
forms a gyrogroup. Further, $\hat{G}$ forms an L-subgyrogroup of
$\sym{G}$. The gyroautomorphism $\gyr{\sig, \tau}{}$ of $\sym{G}$
generated by $\sig$ and $\tau$ is given by
\begin{equation}
\gyr{\sig, \tau}{\lamb} = L_{\gyr{a, b}{c}}\circ \gam
\end{equation}
for all $\lamb = L_c\circ\gam$ in $\sym{G}$.

\begin{theorem}\label{thm: Sym (G) acts on Sym(G)/G}
The gyrogroup $\sym{G}$ whose operation is given by \eqref{eqn:
gyroaddition in Sym (G)} acts transitively on the coset space
$\sym{G}/\hat{G}$ by left gyroaddition.
\end{theorem}
\begin{proof}
Let $\sig = L_a\circ\alp$, $\tau = L_b\circ \bet$, and $\lamb =
L_c\circ\gam$ be arbitrary permutations of $G$, where $a$, $b$,
$c\in G$ and $\alp$, $\bet$, $\gam\in\symz{G}$. For all $x\in G$, we
have $\gyr{\sig, \tau}{L_x} = L_{\gyr{a, b}{x}}$. Hence, $\gyr{\sig,
\tau}{(\hat{G})}\subseteq \hat{G}$. We compute
\begin{align*}
\ominus \lamb \oplus \gyr{\sig, \tau}{\lamb} &=
(L_{\ominus c}\circ\gam^{-1})\oplus(L_{\gyr{a, b}{c}}\circ\gam)\\
{} &= L_{\ominus c\oplus\gyr{a,b}{c}}\circ(\gam^{-1}\circ \gam)\\
{} &= L_{\ominus c\oplus\gyr{a, b}{c}}.
\end{align*}
This proves $\ominus \lamb \oplus \gyr{\sig,
\tau}{\lamb}\in\hat{G}$. By Lemma \ref{lem: Equivalence of gyr is
identity on coset space}, $\gyr{\sig, \tau}{(\lamb\oplus
\hat{G})}\subseteq \lamb\oplus \hat{G}$ for all $\sig$, $\tau$,
$\lamb\in\sym{G}$. By Theorem \ref{thm: Orbit and stabilizer of left
action}, $\sym{G}$ acts transitively on $\sym{G}/\hat{G}$.
\end{proof}

\par Let $\B$ denote the open unit ball of $\R^n$, that is, $\B =
\cset{\vec{v}\in\R^n}{\norm{\vec{v}} < 1}$. Recall that $\B$
equipped with Einstein addition
\begin{equation}\label{eqn: Einstein addition}
\vec{u}\oplus_E\vec{v} =
\dfrac{1}{1+\gen{\vec{u},\vec{v}}}\Bb{\vec{u}+
\dfrac{1}{\gam_{\vec{u}}}\vec{v} +
\dfrac{\gam_{\vec{u}}}{1+\gam_{\vec{u}}}\gen{\vec{u},\vec{v}}\vec{u}},
\end{equation}
where $\gam_{\vec{u}}$ is the Lorentz factor given by
$\gam_{\vec{u}} = \dfrac{1}{\sqrt{1-\norm{\vec{u}}^2}}$, forms a
gyrogroup, the so-called \textit{Einstein gyrogroup}
\cite{AU2007EVA}.

\par Applying Theorem \ref{thm: Sym (G) acts on Sym(G)/G} to the
Einstein gyrogroup $(\B, \oplus_E)$, we have $\sym{\B}$ is a
gyrogroup under the operation given by
\begin{equation}\label{eqn: addition of Sym (B), Einstein}
\sig\oplus\tau = L_{\vec{u}\oplus_E\vec{v}}\circ(\alp\circ\bet)
\end{equation}
for all $\sig = L_{\vec{u}}\circ\alp$, $\tau =
L_{\vec{v}}\circ\bet$, where $\vec{u}$, $\vec{v}\in\B$ and $\alp$,
$\bet\in\symz{\B}$. Further, we have the following theorem.

\begin{theorem}
The gyrogroup $\sym{\B}$ whose operation is given by \eqref{eqn:
addition of Sym (B), Einstein} acts transitively on the coset space
$\sym{\B}/\hat{\B}$ by left gyroaddition.
\end{theorem}

\par Recall also that $\B$ equipped with M\"{o}bius addition
\begin{equation}\label{eqn: Mobius addition}
\vec{u}\oplus_M\vec{v} = \dfrac{(1 + 2\gen{\vec{u},\vec{v}} +
\norm{\vec{v}}^2)\vec{u} + (1 - \norm{\vec{u}}^2)\vec{v}}{1 +
2\gen{\vec{u},\vec{v}} + \norm{\vec{u}}^2\norm{\vec{v}}^2}
\end{equation}
forms a gyrogroup, the so-called \textit{M\"{o}bius gyrogroup}
\cite{AU2008FMG}. Similarly, $\sym{\B}$ is a gyro-group under the
operation given by
\begin{equation}\label{eqn: addition of Sym (B), Mobius}
\sig\oplus\tau = L_{\vec{u}\oplus_M\vec{v}}\circ(\alp\circ\bet)
\end{equation}
for all $\sig = L_{\vec{u}}\circ\alp$, $\tau =
L_{\vec{v}}\circ\bet$, where $\vec{u}$, $\vec{v}\in\B$ and $\alp$,
$\bet\in\symz{\B}$. \mbox{Applying} Theorem \ref{thm: Sym (G) acts
on Sym(G)/G} to the M\"{o}bius gyrogroup $(\B,\oplus_M)$, we have
the \mbox{following} \mbox{theorem}.

\begin{theorem}
The gyrogroup $\sym{\B}$ whose operation is given by \eqref{eqn:
addition of Sym (B), Mobius} acts transitively on the coset space
$\sym{\B}/\hat{\B}$ by left gyroaddition.
\end{theorem}

\section*{Acknowledgments}\addcontentsline{toc}{section}{Acknowledgments}
\par As a visiting researcher, the author is grateful to the
Department of \mbox{Mathematics}, North Dakota State University. He
would like to express his special gratitude to Professor Abraham
Ungar for his hospitality. He is also grateful to the anonymous
referee for his/her careful reading of the manuscript. The financial
support by Institute for Promotion of Teaching Science and
Technology (IPST), Thailand, via Development and Promotion of
Science and Technology Talents Project (DPST), is greatly
appreciated.\\[-1cm]

\section*{}
\end{document}